\documentclass[11pt]{article}
\usepackage{amssymb}
\newtheorem{precor}{{\bf Corollary}}

\newtheorem{precon}{{\bf Conjecture}}

\newtheorem{prealphcon}{{\bf Conjecture}}

\newtheorem{predefin}{{\bf Definition}}

\newtheorem{preexm}{{\bf Example}}

\newtheorem{preappl}{{\bf Application}}

\newtheorem{prelem}{{\bf Lemma}}

\newtheorem{preproof}{{\bf Proof.\ }}

\newenvironment{proof}[1]{\begin{preproof}{\rm
               #1}\hfill{$\blacksquare$}}{\end{preproof}}
\newtheorem{pretheorem}{{\bf Theorem}}

\newenvironment{theorem}{\begin{pretheorem}{\hspace{-0.5
               em}{\bf.\ }}}{\end{pretheorem}}
\newtheorem{prealphtheorem}{{\bf Theorem}}

\newtheorem{prealphlem}{{\bf Lemma}}

\newtheorem{prepro}{{\bf Proposition}}

\newtheorem{preprb}{{\bf Problem}}

\newtheorem{prerem}{{\bf Remark}}

\newtheorem{preapp}{{\bf Application}}

\newtheorem{prequ}{{\bf Question}}

\def\conct[#1,#2]{\mbox {${#1} \leftrightarrow {#2}$}}
\def\dconct[#1,#2]{\mbox {${#1} \rightarrow {#2}$}}
\def\deg[#1,#2]{\mbox {$d_{_{#1}}(#2)$}}
\def\mindeg[#1]{\mbox {$\delta_{_{#1}}$}}
\def\maxdeg[#1]{\mbox {$\Delta_{_{#1}}$}}
\def\outdeg[#1,#2]{\mbox {$d_{_{#1}}^{^+}(#2)$}}
\def\minoutdeg[#1]{\mbox {$\delta_{_{#1}}^{^+}$}}
\def\maxoutdeg[#1]{\mbox {$\Delta_{_{#1}}^{^+}$}}
\def\indeg[#1,#2]{\mbox {$d_{_{#1}}^{^-}(#2)$}}
\def\minindeg[#1]{\mbox {$\delta_{_{#1}}^{^-}$}}
\def\maxindeg[#1]{\mbox {$\Delta_{_{#1}}^{^-}$}}

\def\dre[#1,#2,#3]{\mbox {${\cal E}^{^{#3}}(#1,#2)$}}
\def\var[#1,#2]{\mbox {${\rm Var}_{_{#1}}(#2)$}}
\def\ls[#1]{\mbox {$\xi^{^{#1}}$}}
\def\hom[#1,#2]{\mbox {${\rm Hom}({#1},{#2})$}}
\def\onvhom[#1,#2]{\mbox {${\rm Hom^{v}}(#1,#2)$}}
\def\onehom[#1,#2]{\mbox {${\rm Hom^{e}}(#1,#2)$}}
\def\core[#1]{\mbox {$#1^{^{\bullet}}$}}
\def\cay[#1,#2]{\mbox {${\rm Cay}({#1},{#2})$}}
\def\sch[#1,#2,#3]{\mbox {${\rm Sch}({#1},{#2},{#3})$}}
\def\cays[#1,#2]{\mbox {${\rm Cay_{s}}({#1},{#2})$}}
\def\dirc[#1]{\mbox {$\stackrel{\rightarrow}{C}_{_{#1}}$}}
\def\cycl[#1]{\mbox {${\bf Z}_{_{#1}}$}}

\begin{document}

\begin{center} 
{\Large \bf Regarding Equitable Colorability Defect of Hypergraphs}\\
\vspace{0.3 cm}
{\bf Saeed Shaebani}\\
{\it School of Mathematics and Computer Science}\\
{\it Damghan University}\\
{\it P.O. Box {\rm 36716-41167}, Damghan, Iran}\\
{\tt shaebani@du.ac.ir}\\ \ \\
\end{center}
\begin{abstract}
\noindent Azarpendar and Jafari in 2020 proved the following inequality
$$\chi \left(  {\rm KG} ^r ({\cal F} , s)  \right) \geq \left\lceil  \frac{ {\rm ecd}^r \left(  {\cal F} , \left\lfloor  \frac{s}{2}   \right\rfloor \right) }{r-1}  \right\rceil ,$$
and noted that it is plausible that the above inequality remains true if one replaces $\left\lfloor  \frac{s}{2}   \right\rfloor$ with $s$.

\noindent In this paper, considering the relation 
${\rm ecd}^r \left(  {\cal F} , x \right) \geq {\rm cd}^r \left(  {\cal F} , x   \right)$ which always holds, we show that even in the weaker inequality
$$\chi \left(  {\rm KG} ^r ({\cal F} , s)  \right) \geq \left\lceil  \frac{ {\rm cd}^r \left(  {\cal F} , \left\lfloor  \frac{s}{2}   \right\rfloor \right) }{r-1}  \right\rceil ,$$
no number $x$ greater than $\left\lfloor  \frac{s}{2}   \right\rfloor$ could be replaced by $\left\lfloor  \frac{s}{2}   \right\rfloor$.
\\


\end{abstract}
\section{Introduction}

A {\it hypergraph} ${\cal F}$ consists of a finite set $V( {\cal F} )$, together with a subset of $2^{V ( {\cal F} )} \setminus \{\varnothing\}$ which is denoted by $E( {\cal F} )$.
Any member of $V( {\cal F} )$ is called a {\it vertex} of $ {\cal F} $, and members of $E( {\cal F} )$ are called {\it hyperedges} of $ {\cal F} $.
The hypergraph ${\cal F}$ is called $r$-{\it uniform} whenever $|e|=r$ for each hyperedge $e$ of ${\cal F}$.

Let ${\cal F}$ be an arbitrary (uniform or nonuniform) hypergraph and $r \in \{2,3,4, \dots\}$. If $S$ is a nonnegative integer such that $s < |e|$ for each hyperedge $e$ of ${\cal F}$, then the {\it generalized Kneser Hypergraph} ${\rm KG} ^r ({\cal F} , s)$ is defined as an $r$-uniform hypergraph with vertex set $V \left(  {\rm KG} ^r ({\cal F} , s)  \right) := E( {\cal F} )$ in such a way that $r$ hyperedges $e_{1}, e_{2}, \dots , e_{r}$
of ${\cal F}$ form a hyperedge of ${\rm KG} ^r ({\cal F} , s)$ whenever
$\left|   e_{i} \cap e_{j}   \right| \leq s$ for all distinct indices $i$ and $j$ in $\{1,2, \dots , r\}$. As a definition, the {\it chromatic number} of ${\rm KG} ^r ({\cal F} , s)$, denoted by $\chi \left(  {\rm KG} ^r ({\cal F} , s)  \right)$, is the minimum cardinality of a set $C$ for which a function $f: V \left(  {\rm KG} ^r ({\cal F} , s)  \right) \longrightarrow C$ exists in such a way that
$\left|  \left\{ f\left(e_{1}\right) ,  f\left(e_{2}\right) , \dots ,  f\left(e_{r}\right) \right\} \right| \geq 2$ for each hyperedge
$\left\{ e_{1} ,  e_{2} , \dots ,  e_{r} \right\}$ of
${\rm KG} ^r ({\cal F} , s)$.

For nonnegative integers $n$ and $k$, let the symbols $[n]$ and ${[n] \choose k}$
denote the following sets :
$$\begin{array}{lll}
	[n]:=\{1,2, \dots , n\};  &  {\rm and}   &   {[n] \choose k} := \{A: A\subseteq [n] \ {\rm and} \ |A|=k\}. \\
\end{array}$$
If $n$ is a positive integer and $k$ is a nonnegative integer, then the {\it complete $k$-uniform hypergraph} $K_{n}^{k}$, is a hypergraph with $V \left( K_{n}^{k} \right) := [n]$ and $E \left( K_{n}^{k} \right) := {[n] \choose k}$.

One can easily observe that 
$\chi \left(  {\rm KG}^2 \left(K_{2n+k}^n , 0 \right)  \right)  \leq k+2$
and
$\chi \left(  {\rm KG}^r \left(K_{n}^k , 0 \right)  \right)  \leq \left\lceil \frac{n-r(k-1)}{r-1} \right\rceil$.
In 1955, Kneser \cite{MR0068536} conjectured that
$\chi \left(  {\rm KG}^2 \left(K_{2n+k}^n , 0 \right)  \right)  = k+2$.
This conjecture was settled by Lov{\'a}sz \cite{MR514625} in 1978. Later, in 1986, Alon, Frankl, and Lov{\'a}sz \cite{MR857448} proved a conjecture of Erd{\H{o}}s \cite{MR0465878}
which asserts that if $n\geq r(k-1)+1$ then
$\chi \left(  {\rm KG}^r \left(K_{n}^k , 0 \right)  \right)  = \left\lceil \frac{n-r(k-1)}{r-1} \right\rceil$.

If $A$ and $B$ are two sets and $s$ is a nonnegative integer, then as a notation, we write $A \subseteq _{s} B$ whenever $|A \setminus B| \leq s$.

\noindent Let ${\cal F}$ be a hypergraph, and $r$ and $s$ be nonnegative integers such that $r\geq 2$ and $s < |e|$ for each hyperedge $e$ of ${\cal F}$. As a definition, the
{\it $s$-th $r$-colorability defect} of ${\cal F}$, denoted by ${\rm cd}^r ({\cal F} , s)$, is equal to the minimum size of a subset $X_{0}$ of $V({\cal F})$ for which
a partition $\{X_1 , X_2 , \dots , X_r\}$ of $V({\cal F}) \setminus X_0$ exists such that :
\begin{center}
	If $e \in E({\cal F})$ and $X_i \in \{X_1 , X_2 , \dots , X_r\}$, then $e \nsubseteq _s X_i.$
\end{center}
We note that in this definition, some of $X_1 , X_2 , \dots , X_r$ may be equal to the empty set.

\noindent Dolnikov \cite{MR953021} and K{\v{r}}{\'{\i}}{\v{z}} \cite{MR1081939, 10.2307/118094} proved that for any hypergraph ${\cal F}$ and each integer $r\geq 2$ we have
$$\chi \left(  {\rm KG} ^r ({\cal F} , 0)  \right) \geq \left\lceil  \frac{ {\rm cd}^r \left(  {\cal F} , 0 \right) }{r-1}  \right\rceil .$$

\noindent It is easily seen that
${\rm cd}^2 \left(K_{2n+k}^n , 0 \right)  = k+2$.
Also, it is evident that if $n\geq r(k-1)+1$ then
${\rm cd}^r \left(K_{n}^k , 0 \right)  = n-r(k-1)$.
Therefore, the theorem of Dolnikov and K{\v{r}}{\'{\i}}{\v{z}} is a generalization of theorems of Lov{\'a}sz \cite{MR514625} and Alon, Frankl, and Lov{\'a}sz \cite{MR857448}.

If ${\cal F}$ is a hypergraph
and $r$ and $s$ are nonnegative integers such that $r \geq 2$ and $s < |e|$ for each hyperedge $e$ of ${\cal F}$, then the {\it $s$-th equitable $r$-colorability defect} of ${\cal F}$,
denoted by ${\rm ecd} ^r ({\cal F} , s)$, is defined as the minimum cardinality of a subset $X_0$ of $V({\cal F})$ in such a way that a partition
$\{X_1 , X_2 , \dots , X_r\}$ of $V({\cal F}) \setminus X_0$ with the two following properties exists :
\begin{itemize}
	\item If $1 \leq i < j \leq r$ then $\left| \left|X_i\right| - \left|X_j\right| \right| \leq 1;$
	\item If $e \in E({\cal F})$ and $X_i \in \{X_1 , X_2 , \dots , X_r\}$, then $e \nsubseteq _s X_i.$
\end{itemize}
We note that in this definition, some of $X_1 , X_2 , \dots , X_r$ could be equal to the empty set. Obviously, ${\rm ecd} ^r ({\cal F} , s) \geq {\rm cd} ^r ({\cal F} , s)$.
As a generalization of the theorem of Dolnikov and K{\v{r}}{\'{\i}}{\v{z}}, it was proved by Abyazi Sani and Alishahi \cite{ABYAZISANI2018229} that the relation
$$\chi \left(  {\rm KG} ^r ({\cal F} , 0)  \right) \geq \left\lceil  \frac{ {\rm ecd}^r \left(  {\cal F} , 0 \right) }{r-1}  \right\rceil $$
always holds.

Azarpendar and Jafari \cite{https://doi.org/10.48550/arxiv.2009.05969} as a generalization of many earlier results \cite{ABYAZISANI2018229, MR857448, AZARPENDAR2021372, MR953021, MR1081939, 10.2307/118094, MR514625} proved the following theorem.
\begin{theorem} \cite{https://doi.org/10.48550/arxiv.2009.05969}
	Let ${\cal F}$ be a hypergraph. If $r$ and $s$ are nonnegative integers such that
	$r\geq 2$ and $s < |e|$ for each hyperedge $e$ of {\cal F}, then
	$$\chi \left(  {\rm KG} ^r ({\cal F} , s)  \right) \geq \left\lceil  \frac{ {\rm ecd}^r \left(  {\cal F} , \left\lfloor  \frac{s}{2}   \right\rfloor \right) }{r-1}  \right\rceil .$$
\end{theorem}

Azarpendar and Jafari in \cite{https://doi.org/10.48550/arxiv.2009.05969} noted that it is plausible that the above theorem remains true if one replaces $\left\lfloor  \frac{s}{2}   \right\rfloor$ with $s$.

\noindent In this paper, as our first result, we show that the inequality
$$\chi \left(  {\rm KG} ^r ({\cal F} , s)  \right) \geq \left\lceil  \frac{ {\rm ecd}^r \left(  {\cal F} , s \right) }{r-1}  \right\rceil $$
does not hold in general.

We know that
$${\rm ecd}^r \left(  {\cal F} , 0 \right) \leq
{\rm ecd}^r \left(  {\cal F} , 1 \right) \leq
\dots \leq
{\rm ecd}^r \left(  {\cal F} , \left\lfloor  \frac{s}{2}   \right\rfloor \right) \leq
\dots \leq
{\rm ecd}^r \left(  {\cal F} , s \right) .
$$
One may ask whether the inequality
$$\chi \left(  {\rm KG} ^r ({\cal F} , s)  \right) \geq \left\lceil  \frac{ {\rm ecd}^r \left(  {\cal F} , \left\lfloor  \frac{s}{2}   \right\rfloor \right) }{r-1}  \right\rceil $$
is still true if we put some other values larger than $\left\lfloor  \frac{s}{2}   \right\rfloor$ instead of $\left\lfloor  \frac{s}{2}   \right\rfloor$.
In order to answer this natural question, we consider the relation
 $${\rm cd}^r \left(  {\cal F} , 0 \right) \leq
 {\rm cd}^r \left(  {\cal F} , 1 \right) \leq
 \dots \leq
 {\rm cd}^r \left(  {\cal F} , \left\lfloor  \frac{s}{2}   \right\rfloor \right) \leq
 \dots \leq
 {\rm cd}^r \left(  {\cal F} , s \right) ,
 $$
and also the relation
${\rm ecd}^r \left(  {\cal F} , x \right) \geq {\rm cd}^r \left(  {\cal F} , x   \right)$ which always holds. As our second result, we show that even in the weaker inequality
$$\chi \left(  {\rm KG} ^r ({\cal F} , s)  \right) \geq \left\lceil  \frac{ {\rm cd}^r \left(  {\cal F} , \left\lfloor  \frac{s}{2}   \right\rfloor \right) }{r-1}  \right\rceil ,$$
no number $x$ greater than $\left\lfloor  \frac{s}{2}   \right\rfloor$ could be replaced by $\left\lfloor  \frac{s}{2}   \right\rfloor$.

\section{We cannot replace $s$ instead of $\left\lfloor  \frac{s}{2}   \right\rfloor$.}

In this section, our aim is showing that the inequality
$$\chi \left(  {\rm KG} ^r ({\cal F} , s)  \right) \geq \left\lceil  \frac{ {\rm ecd}^r \left(  {\cal F} , s \right) }{r-1}  \right\rceil $$
is not correct. If we put $r=2$, then the expression $\left\lceil  \frac{ {\rm ecd}^r \left(  {\cal F} , s \right) }{r-1}  \right\rceil $ will be equal to ${\rm ecd}^2 \left(  {\cal F} , s \right)$. Hence, in order to disprove 
$\chi \left(  {\rm KG} ^r ({\cal F} , s)  \right) \geq \left\lceil  \frac{ {\rm ecd}^r \left(  {\cal F} , s \right) }{r-1}  \right\rceil $, it is enough to find a hypergraph for which
$$\chi \left(  {\rm KG} ^2 ({\cal F} , s)  \right) < {\rm ecd}^2 \left(  {\cal F} , s \right) .$$
In this regard, we state and prove the following theorem, which is the first result of this paper.

\begin{theorem}
	For any two positive integers $l$ and $s$ with $l\geq 2$, there exists
	a hypergraph ${\cal F}$ such that
	$$\begin{array}{lcr}
		$$\chi \left(  {\rm KG} ^2 ({\cal F} , s)  \right) =l & {\rm and} & {\rm ecd}^2 \left(  {\cal F} , s \right) = l+s.$$
	\end{array}$$
\end{theorem}

\begin{proof}{
Put $k:=l-2$. Let $K_{2n+k}^{n}$ be the hypergraph with vertex set $[2n+k]$ and hyperedge set ${[2n+k] \choose n}$. So,
$$\chi \left(  {\rm KG}^2 \left(K_{2n+k}^n , 0 \right)  \right)  = k+2=l.$$
Now, let $S:=\{a_1 , a_2 , \dots , a_s\}$ be a set such that $S \cap [2n+k] = \varnothing $;
and then define a new hypergraph ${\cal F}$ with vertex set $V({\cal F}) := S \cup [2n+k]$
whose hyperedge set equals
$$E({\cal F}) := \left\{e \cup S : e \in E \left(K_{2n+k}^n  \right) \right\} = \left\{e \cup S : e \in {[2n+k] \choose n} \right\}.$$

Any two vertices $e_1$ and $e_2$ from ${\rm KG}^2 \left(K_{2n+k}^n , 0 \right)$
are adjacent iff their corresponding vertices in ${\rm KG}^2 \left( {\cal F} , s \right)$
are adjacent; that is, $\left\{ e_1, e_2 \right\} \in E \left(  {\rm KG}^2 \left(K_{2n+k}^n , 0 \right)  \right)$ if and only if $\left\{ e_1 \cup S , e_2 \cup S \right\} \in E \left(  {\rm KG}^2 \left( {\cal F} , s \right)  \right)$.
So, we observe that
$$\chi \left(  {\rm KG} ^2 ({\cal F} , s)  \right) = \chi \left(  {\rm KG}^2 \left(K_{2n+k}^n , 0 \right)  \right)  = k+2=l .$$
Now, we show that
${\rm ecd}^2 \left(  {\cal F} , s \right) = k+2+s = l+s.$
In this regard, our first objective is showing that
${\rm ecd}^2 \left(  {\cal F} , s \right) \leq k+2+s .$
Put
\begin{itemize}
	\item $Y_1 := [n-1] = \{1 , 2 , \dots , n-1\} ;$
	\item $Y_2 := [2n-2] \setminus [n-1] = \{n , n+1 , \dots , 2n-2\} ;$
	\item $Y_0 := \{2n-1 , 2n , 2n+1 , \dots , 2n+k\} \cup S .$
\end{itemize}
Obviously, $\left\{ Y_0, Y_1 , Y_2 \right\}$ is a partition of $V({\cal F})$ such that $\left| Y_1 \right| = \left| Y_2 \right| = n-1$. Also, if $e \in E({\cal F})$ and $i \in \{1,2\}$ then :
$$
	\left| e \setminus Y_i  \right|  =  \left| e \setminus \left( e \cap Y_i \right)  \right|  =  \left| e \right| - \left| e \cap Y_i  \right|  \geq  \left| e \right| - \left| Y_i  \right| \\
	 =  (s+n) - (n-1)  =  s+1   >  s;
$$
and therefore, $e \nsubseteq _s Y_i$.
We conclude that
${\rm ecd}^2 \left(  {\cal F} , s \right) \leq \left| Y_0  \right| = k+2+s .$

As our next task, we aim to prove that
${\rm ecd}^2 \left(  {\cal F} , s \right) \geq k+2+s .$
Suppose, on the contrary, that
${\rm ecd}^2 \left(  {\cal F} , s \right) \leq k+1+s .$
So, one can regard a partition $\left\{ X_0, X_1 , X_2 \right\}$ of $V({\cal F})$
with $\left| X_0 \right| = {\rm ecd}^2 \left(  {\cal F} , s \right) \leq k+1+s$ and
$\left| X_1 \right| \geq \left| X_2 \right|$ in such a way that
$e \nsubseteq _s X_1$ and $e \nsubseteq _s X_2$ for each hyperedge $e$ in $E({\cal F})$.
Hence,
$$
	2\left| X_1 \right|  \geq  \left| X_1 \right| + \left| X_2 \right|  =  |V({\cal F})|
	- \left| X_0 \right| 
	    \geq  (2n+k+s) - (k+1+s)  =  2n-1;
$$
and therefore, $$\left| X_1 \right| \geq \left\lceil  \frac{2n-1}{2}  \right\rceil = n .$$
Choose a subset $X'_1$ of $X_1$ such that $\left| X'_1 \right| = n $.
Define
$$\begin{array}{lcr}
	A:= X'_1 \cap [2n+k] & {\rm and}  & B:= X'_1 \cap S .
\end{array}$$
So, $|A| + |B| = \left| X'_1 \right| = n$.
Also, suppose that
$$\begin{array}{lcr}
	[2n+k] \setminus X'_1 = \left\{ i_1 , i_2 , \dots , i_{2n+k-|A|}\right\} &
	 {\rm and} & i_1 < i_2 < \dots < i_{2n+k-|A|}.
\end{array}$$

\noindent Now, $e:= X'_1 \cup \left\{ i_1 , i_2 , \dots , i_{n-|A|}\right\} \cup \left( S \setminus X'_1 \right)$
is a hyperedge of ${\cal F}$ that satisfies
$$\begin{array}[pos]{lllll}
	\left| e \setminus X_1  \right| & \leq & \left| e \setminus X'_1  \right| & = &  \left| \left\{ i_1 , i_2 , \dots , i_{n-|A|}\right\} \cup \left( S \setminus X'_1 \right)  \right|  \\
	    &  &  &  &  \\
	    &   &    &  =  &  \left| \left\{ i_1 , i_2 , \dots , i_{n-|A|}\right\}   \right| +   \left|   S \setminus X'_1   \right|  \\
	    &  &  &  &    \\
	    &   &   &  =  & n-|A| + \left| S \setminus \left( S \cap X'_1 \right)  \right| \\
	    &  &  &  &  \\
	    &   &   &  =  & n-|A| +   |S| - \left|  S \cap X'_1   \right|   \\     
	    &  &  &  &  \\
	    &   &  & = & n- |A| + |S| - |B|  \\
	    &  &  &   &  \\
	    &   &   & = & |S| + n - (|A| + |B|)    = |S| = s . \\
\end{array}$$
We conclude that $\left| e \setminus X_1  \right|  \leq s$,
a contradiction to the fact that $e \nsubseteq _s X_1$.
It follows that
${\rm ecd}^2 \left(  {\cal F} , s \right) \geq k+2+s .$
Therefore, ${\rm ecd}^2 \left(  {\cal F} , s \right) \geq k+2+s $
and ${\rm ecd}^2 \left(  {\cal F} , s \right) \leq k+2+s $
imply ${\rm ecd}^2 \left(  {\cal F} , s \right) = k+2+s = l + s $;
as desired.
}
\end{proof}

\section{A stronger result}

This section concerns with determining the set of values which can be replaced by
$\left\lfloor  \frac{s}{2}   \right\rfloor$
in the general inequality
$\chi \left(  {\rm KG} ^r ({\cal F} , s)  \right) \geq \left\lceil  \frac{ {\rm ecd}^r \left(  {\cal F} , \left\lfloor  \frac{s}{2}   \right\rfloor \right) }{r-1}  \right\rceil .$
Since
$${\rm ecd}^r \left(  {\cal F} , 0 \right) \leq
{\rm ecd}^r \left(  {\cal F} , 1 \right) \leq
\dots \leq
{\rm ecd}^r \left(  {\cal F} , \left\lfloor  \frac{s}{2}   \right\rfloor \right) \leq
\dots \leq
{\rm ecd}^r \left(  {\cal F} , s \right) ,
$$
one observes that each nonnegative integer which is less than or equal to
$\left\lfloor  \frac{s}{2}   \right\rfloor$,
could be replaced by $\left\lfloor  \frac{s}{2}   \right\rfloor$.
The aim of this section is showing that no number $x$ greater than $\left\lfloor  \frac{s}{2}   \right\rfloor$ could be replaced by $\left\lfloor  \frac{s}{2}   \right\rfloor$. In this regard, we consider the relation
${\rm ecd}^r \left(  {\cal F} , x \right) \geq {\rm cd}^r \left(  {\cal F} , x   \right)$ which always holds; and we show that even in the weaker inequality
$$\chi \left(  {\rm KG} ^r ({\cal F} , s)  \right) \geq \left\lceil  \frac{ {\rm cd}^r \left(  {\cal F} , \left\lfloor  \frac{s}{2}   \right\rfloor \right) }{r-1}  \right\rceil ,$$
no number $x$ greater than $\left\lfloor  \frac{s}{2}   \right\rfloor$ could be replaced by $\left\lfloor  \frac{s}{2}   \right\rfloor$.
We restrict our attention just to the case where $r=2$ and $s$ is an even positive integer.

\begin{theorem}
	Let $k \in \mathbb{N}$ and $s$ be an even positive integer. Then, there exists a hypergraph
	${\cal F}$ with $\chi \left(  {\rm KG} ^2 ({\cal F} , s)  \right) = k$
	in such a way that
	$${\rm cd}^2 \left(  {\cal F} , l \right) = {\rm ecd}^2 \left(  {\cal F} , l \right) =
	k (2l - s +1) $$
	for each $l$ in $\left\{\frac{s}{2} + 1 , \frac{s}{2} + 2 , \dots , s \right\}$.
\end{theorem}

\begin{proof}{
	Let us regard some pairwise disjoint sets
	$A_1 , A_2 , \dots , A_k$
	with
	$$\left| A_1 \right| = \left| A_2 \right| = \dots = \left| A_k \right| = s+1.$$
	Now, define a hypergraph ${\cal F}$ with
	$$\begin{array}{lcr}
		V ({\cal F}) := A_1 \cup A_2 \cup \dots \cup A_k  &  {\rm and}   &   E ({\cal F}) := \left\{   A_1 , A_2 , \dots , A_k \right\}.
	\end{array}$$
	Let $\left\{ X_0 , X_1 , X_2 \right\}$ be a partition of $V ({\cal F})$ that satisfies the following two properties :
	\begin{itemize}
		\item $\left| X_0  \right| = {\rm cd}^2 \left(  {\cal F} , l \right) ,$
		\item If $A_i \in \left\{   A_1 , A_2 , \dots , A_k \right\}$ and $X_j \in \left\{ X_1 , X_2 \right\}$ , then $A_i \nsubseteq _l X_j$ .
	\end{itemize}
	We aim to show that $\left| X_0 \right| \geq k(2l-s+1)$.
	
	\noindent Let $i \in \{1,2, \dots , k\}$ and $j \in \{1,2\}$.
	Since $A_i \nsubseteq _l X_j$, we must have
	$$\left|  A_i \setminus X_j\right| \geq l+1 .$$ Hence,
	$$\left|  A_i \cap X_j\right| = \left|  A_i \right| - \left|  A_i \setminus X_j\right| = (s+1) - \left|  A_i \setminus X_j\right| \leq (s+1) - (l+1) = s-l .$$
	Thus, $\left|  A_i \cap X_j\right|  \leq s-l $ for each $i$ in $\{1,2, \dots , k\}$ and each $j$ in $\{1,2\}$. Hence,
	$$\left| X_1 \right| = \left| X_1 \cap V({\cal F}) \right| =  \left| X_1 \bigcap \left( \bigcup_{i=1}^k A_i \right) \right| \leq \sum _{i=1} ^k  \left| X_1 \cap A_i \right| \leq  \sum _{i=1} ^k  (s-l) = k(s-l) .$$
	Similarly, $\left| X_2 \right| \leq  k(s-l) .$
	We conclude that
	$$\left| X_1 \cup X_2 \right| = \left| X_1 \right| + \left| X_2 \right| \leq 2k(s-l) . $$
	Therefore,
	$$\left| X_0 \right| = \left| V({\cal F}) \right| - \left| X_1 \cup X_2 \right| = k(s+1) - \left| X_1 \cup X_2 \right| \geq k(s+1) - 2k(s-l) = k(2l-s+1) . $$
	We conclude that
	$${\rm cd}^2 \left(  {\cal F} , l \right) \geq k (2l - s +1) .$$
	Now, we claim that ${\rm ecd}^2 \left(  {\cal F} , l \right) \leq k (2l - s +1) .$
	In this regard, for each $i$ in $\{1, \dots , k\}$, let $A_{i_1}$ and $A_{i_2}$
	be two disjoint subsets of $A_i$, each of size $s-l$.
	More precisely,
	$$\begin{array}{ccccc}
		A_{i_1} \cup A_{i_2}  \subseteq A_i   &  {\rm and}   &  A_{i_1} \cap A_{i_2} = \varnothing  ,  &  {\rm and \ also,}  &  \left| A_{i_1}  \right| = \left| A_{i_2}  \right| = s-l .
	\end{array}$$
	Now, define a partition $\left\{  Y_0 , Y_1 , Y_2 \right\}$ of $V({\cal F})$ as follows :
	$$Y_1 := \bigcup_{i=1} ^k A_{i_1}  \ \   {\rm and}  \ \  Y_2 := \bigcup _{i=1} ^k A_{i_2}  \ \  {\rm and}  \ \   Y_0 := V({\cal F}) \setminus \left( Y_1 \cup Y_2 \right)  . $$
	We have $Y_1 \cap Y_2 = \varnothing$ and $\left| Y_1 \right| = \left| Y_2 \right| = k(s-l)$. Also,
	$$\left| Y_0 \right| = \left| V({\cal F}) \right| - \left| Y_1 \cup Y_2 \right| =
	k(s+1) - 2k(s-l) = k(2l-s+1) .$$
	Also, if $A_i \in \left\{ A_1 , A_2 , \dots , A_k \right\}$ then because of
	$A_i \cap Y_1 = A_{i_1}$
	we have
	$$\left| A_i \setminus Y_1 \right| = \left| A_i \setminus \left( A_i \cap Y_1 \right) \right| =  
	\left| A_i \right| - \left| A_{i_1} \right| =
	(s+1) - (s-l) = l+1 .$$
	Thus, $\left| A_i \setminus Y_1 \right| = l+1$. Hence, $\left| A_i \setminus Y_1 \right| \nleq l$; and therefore, $A_i \nsubseteq_l Y_1$.
	Similarly, we have $A_i \nsubseteq_l Y_2$.
	
	\noindent It follows that $\left\{  Y_0 , Y_1 , Y_2 \right\}$ is a partition of $V({\cal F})$ that satisfies the following three properties :
	\begin{itemize}
		\item $\left| Y_0 \right| = k(2l-s+1) $;
		\item $\left| Y_1 \right| = \left| Y_2 \right| $;
		\item $e \nsubseteq_l Y_1$ and $e \nsubseteq_l Y_2$ for each hyperedge $e$ of $V({\cal F})$.
	\end{itemize}
	So, ${\rm ecd}^2 \left(  {\cal F} , l \right) \leq \left| Y_0 \right| = k(2l-s+1)$; and therefore, ${\rm ecd}^2 \left(  {\cal F} , l \right) \leq k(2l-s+1)$; as claimed.

    We conclude that $k(2l-s+1) \leq {\rm cd}^2 \left(  {\cal F} , l \right) \leq {\rm ecd}^2 \left(  {\cal F} , l \right) \leq k(2l-s+1)$; which implies
    $$ {\rm cd}^2 \left(  {\cal F} , l \right) = {\rm ecd}^2 \left(  {\cal F} , l \right) = k(2l-s+1) ;$$
    which is desired.	
}	
\end{proof}

\bibliographystyle{plain}
\def\cprime{$'$} \def\cprime{$'$}

\end{document}